\documentclass[12pt]{amsart}
\usepackage[cp1251]{inputenc}
\usepackage{amsmath,amssymb,amsthm}
\usepackage[mathscr]{eucal}
\usepackage{graphicx, epsfig}

\usepackage{tikz}

\theoremstyle{plain}
\newtheorem{theorem}{Theorem}

\newtheorem{lemma}{Lemma}[section]
\newtheorem{remark}{Remark}
\newtheorem{corollary}{Corollary}
\newtheorem{definition}{Definition}

\author{E.A.~Fominykh}
\author{V.G.~Turaev}
\author{A.Yu.~Vesnin}

\address{Chelyabinsk State University and Krasovskii Institute of Mathematics and Mechanics}
\email{fominykh@csu.ru}

\address{Indiana University and Chelyabinsk State University}
\email{vturaev@yahoo.com}

\address{Sobolev Institute of Mathematics and Chelyabinsk State University}
\email{vesnin@math.nsc.ru}

\thanks{This work was supported by the Laboratory of Quantum Topology, Chelyabinsk State University (the Russian government grant no. 14.Z50.31.0020). A.Yu.~Vesnin and E.A.~Fominykh were supported in part by the Russian Foundation for Basic Research (project no. 16-01-00609), and by the Ministry of Education and Science of the Russia (the state task number 1.1260.2014/K).}

\title{Complexity of virtual 3-manifolds}

\begin{document}

\begin{abstract}
 Virtual $3$-manifolds were introduced by S.V. Matveev in 2009 as   natural generalizations of the classical $3$-manifolds. In this paper, we introduce a notion of complexity of a virtual $3$-manifold. We investigate the values of the complexity for virtual 3-manifolds presented by special polyhedra with one or two $2$-components. On the basis of these results, we establish the exact values of the complexity for a wide class of hyperbolic $3$-manifolds with totally geodesic boundary.

\textbf{Key words:} virtual manifolds, $3$-manifolds, hyperbolic manifolds, complexity.
\end{abstract}

\maketitle

\section{Introduction}

  S.V. Matveev~\cite{Mat2009} introduced a notion of a virtual $3$-manifold generalizing the classical
  notion of a $3$-manifold. A virtual manifold is an equivalence class of so-called special polyhedra. Each virtual manifold determines a $3$-manifold with nonempty boundary and, possibly, $RP^2$-singularities. Many properties and invariants of  $3$-manifolds, such as Turaev--Viro invariants~\cite{TV},   extend  to   virtual manifolds.

 The complexity of a manifold is an important invariant in the theory of $3$-manifolds, see~\cite{MatBook}. The problem of calculating of the complexity   is very difficult. The exact values of the
complexity are presently known only for a finite number of tabulated manifolds~\cite{Mat05, FrigMarPet04} and for several infinite families of manifolds with boundary~\cite{FrigMarPet03, VesFom11, VesFom12, Stekl2015},   closed manifolds~\cite{JacoRubTil09, JacoRubTil11} and   manifolds with cusps~\cite{Ves-Tar-Fom2014-1,Ves-Tar-Fom2014-2, FGGTV}. Estimates of the complexity for some infinite families of manifolds are obtained in \cite{MPV, PV, CMV, Fominykh11, FomWiest13,VF2014}.

 In this paper we introduce  the complexity of a virtual $3$-manifold. We investigate the   values of the complexity for virtual manifolds presented by special polyhedra with one or two $2$-components. On the basis of these results we establish the exact values of the
 complexity for a wide class of hyperbolic $3$-manifolds with totally geodesic boundary.

 The paper is organized as follows. In Section~\ref{section2}, we present  basic facts of the theory of simple and special polyhedra, define a virtual  manifold and discuss invariants of virtual manifolds. In Section~\ref{section3}, we introduce the notion of complexity of a virtual manifold, investigate some of its properties and establish two-sided estimates for the complexity. We also state Theorem~\ref{thm:main} and establish Lemma~\ref{lemma:complexity_properties} devoted to a computation of the complexity of virtual manifolds presented by special polyhedra with two $2$-components. Section~\ref{section4} is completely devoted to the proof of Theorem~\ref{thm:main}. For the reader's convenience we include definitions of the $\varepsilon$-invariant and Turaev--Viro invariants, which play a key role in the proof. In Section~\ref{section5}, we discuss a relationship between our results about the complexity of virtual manifolds and the complexity  of classical $3$-manifolds. We also prove Theorem~\ref{theorem9}, giving exact values of the complexity for hyperbolic $3$-manifolds with totally geodesic boundary that have special spines with two $2$-components. In Section~\ref{section6}, we give examples of two infinite families of $3$-manifolds satisfying the conditions of Theorem~\ref{theorem9}.

\section{Virtual 3-manifolds}
 \label{section2}

 Recall some basic facts from the theory of simple and special polyhedra developed by S.V. Matveev (see \cite[Ch.~1]{MatBook}). A compact two-dimensional polyhedron $P$ is said to be \emph{simple} if the link of each point $x \in P$ is homeomorphic to either a circle (such a point $x$ is called \emph{nonsingular}), or a graph consisting of two vertices and three  joining them edges (such $x$ is called a \emph{triple point}), or the complete graph $K_4$ with four vertices (such a point $x$ is called a \emph{true vertex}). The connected components of the set of all nonsingular points are called $2$-\emph{components} of $P$, while the connected components of the set of all triple points are called \emph{triple lines} of $P$. The set of singular points of   $P$ (that is, the union of  triple lines and true vertices) is called the \emph{singular graph} of $P$. Every simple polyhedron is naturally stratified. In this stratification each 2-dimensional stratum   (a $2$-component) is a connected component of the set of nonsingular points, strata of dimension $1$ are the  (open or closed) triple lines, and  strata of dimension $0$ are the true vertices.

 A simple polyhedron is  \emph{special} if each its 1-dimensional stratum is an open $1$-cell, and each its $2$-component is an open $2$-cell. A singular graph of a special polyhedron has at least one true vertex and   is a $4$-regular graph. Therefore, it is natural to call the triple lines of a special polyhedron \emph{edges}.

 For a special polyhedron $P$  with at least two true vertices,   a \emph{$T$-move} on~$P$   removes a proper subpolyhedron of $ P$  shown in Fig. \ref{fig1}  on the left  and glues instead a proper subpolyhedron  shown in Fig. \ref{fig1}  on the right. Note that $T$ increases by 1 both the number of true vertices  and the number of $2$-components of $P$, while the inverse move $T^{-1}$ decreases both these numbers by 1. The moves $T$ and $T^{-1}$ do not change the Euler characteristic of the special polyhedron.

\begin{figure}[h]
\begin{center}
\includegraphics[scale=0.75]{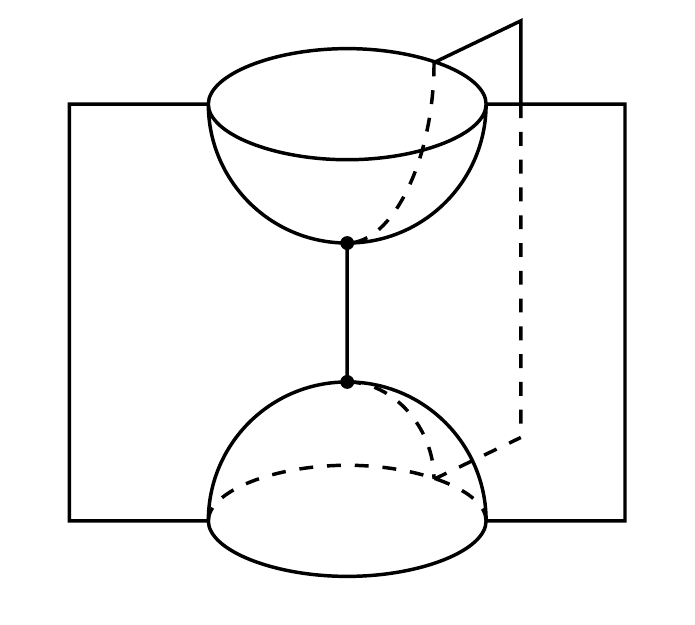}
 \raisebox{70pt}
 {$\begin{array}{c}
      \stackrel T\Longrightarrow\\{}\\\stackrel{\hskip1emT^{-1}}\Longleftarrow
  \end{array}$}
\includegraphics[scale=0.75]{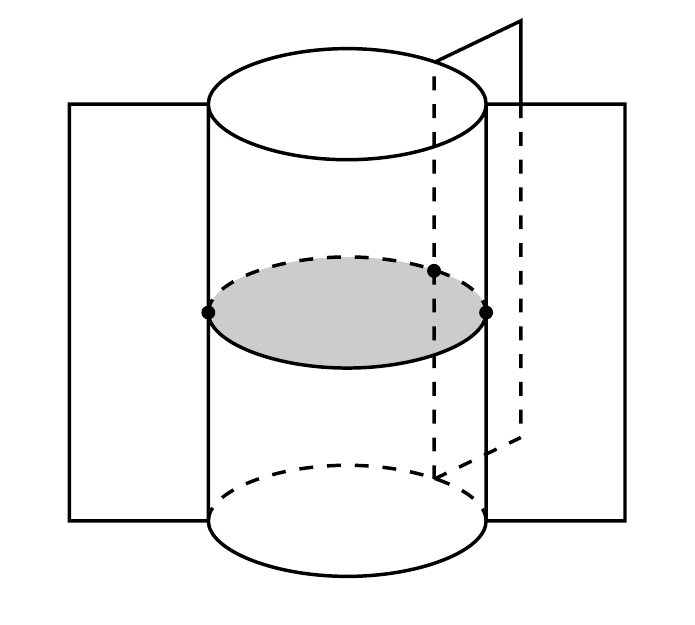}
\end{center}
\caption{The moves $T$ and $T^{-1}$.} \label{fig1}
\end{figure}

 For each $2$-component $\xi$ of a special polyhedron $P$, there is a characteristic map $f: D^{2} \to P$, which carries the interior of the disc $D^{2}$ onto $\xi$ homeomorphically and which restricts  to a local embedding on   $S^{1} = \partial D^{2}$. We will call the curve $f|_{\partial D^{2}} : {\partial D^{2}} \to P$ (and its image $f|_{\partial D^{2}} (\partial D^{2})$) the \emph{boundary curve} $\partial \xi$ of $\xi$. Denote by $A^2 \cup D^{2}$ the annulus $S^{1} \times I$ with a disc $D^{2}$ attached along its middle circle, and denote by $M^2 \cup D^{2}$ a M\"obius band with a disc $D^{2}$ attached along
its middle circle. We say that the boundary curve $\partial \xi$ has a \emph{trivial} (respectively,  \emph{nontrivial})  normal bundle  if the characteristic
map $f : D^{2} \to P$  of~$\xi$   extends to a local embedding $f^{(A^2\cup D^{2})} : A^2 \cup D^{2} \to P$  (respectively, to a local embedding $f^{(M^2\cup D^{2})} : M^2 \cup D^{2} \to P$).
 We say that the boundary curve $\partial \xi$ is \emph{short}, if it traverses precisely three true vertices of $P$ and passes through each of them only once.

We now state a condition on a special polyhedron  allowing to apply to it the move $T^{-1}$.

\begin{remark}
 \label{remark:adm_T^{-1}}
 A special polyhedron allows a move $T^{-1}$ if and only if   at least one of its boundary curves is short and has   a trivial normal bundle.
\end{remark}

 The notion of a virtual manifold was introduced in \cite{Mat2009}.

\begin{definition}
  \label{def:virt_manifold}
 We say that two special polyhedra are equivalent if one of them can be transformed into the other one  by a finite sequence of moves $T^{\pm1}$. A  virtual 3-manifold is the equivalence class $[P]$ of a special polyhedron~$P$.
\end{definition}

 The moves $T$ and $T^{- 1}$ do not change the Euler characteristic   of a special polyhedron  and the number of its $2$-components, whose boundary curves   have non-trivial normal bundles. This implies the following lemma.
\begin{lemma}
 \label{lemma:number_of_nontrivial}
  The Euler characteristic $\chi(P) $  of a special polyhedron $P$ and the number of $2$-components of $P$, whose boundary curves  have non-trivial normal bundles, are invariants of the virtual $3$-manifold $[P]$.
\end{lemma}

 Each special polyhedron $P$ determines a $3$-manifold $W(P)$ with nonempty boundary and $RP^2$-singularities (see \cite{Mat2009}, \cite[\S 1.1.5 ]{MatBook}).

\begin{definition}
 \label{def:sing_manifold}
 A compact three-dimensional polyhedron $W$ is called a  $3$-manifold with $RP^2$-singularities  if the link of any point of $W$ is either a $2$-sphere, or a $2$-disc, or $RP^2$. The set of points of $W$ whose links are $2$-discs form the boundary $\partial W$ of $W$.
\end{definition}

 Let us describe a construction of the manifold $W(P)$ following \cite{Mat2009}. Replacing each true vertex of $P$ by a $3$-ball and each triple line of $P$ by a handle of index $1$, we obtain a (possibly, non-orientable) handlebody $H$ such that each $2$-component $\xi$ of $P$ meets $\partial H$ along a single circle $c=\xi\cap \partial H$. The rest of the  construction involves   two cases, depending on the type of the normal bundle of the boundary curve $\partial \xi$. If the normal bundle of $\partial \xi$ is trivial,  then a regular
neighborhood $N(c)$ of   $c$ in $\partial H$ is an annulus. In this
case we thicken the disc $\xi\setminus (\xi\cap H)$   to an index $2$ handle. If the normal bundle of $\partial \xi$ is non-trivial,  then a regular neighborhood $N(c)$ of   $c$ in $\partial H$ is a M\"obius strip. In this case we attach  to the boundary $\partial H$ of $H$ a disc $D^2$ along the circle $\partial N(c)$ and replace   $\xi\setminus (\xi\cap H)$ with a cone over the projective space $RP^2=N(c) \cup D^2$. Doing so for all $2$-components of $P$, we obtain a $3$-manifold $W(P)$ with nonempty boundary and $RP^2$-singularities.

 The following theorem shows that each virtual manifold determines a $3$-manifold with  $RP^2$-singularities and nonempty boundary.

\begin{theorem} \cite[Theorem 3]{Mat2009}
 \label{thm:surjection}
 The assignment $P \to W(P)$ yields a well defined surjection
$\varphi : \mathcal V \to \mathcal W$ from the set $\mathcal V$ of all virtual $3$-manifolds onto the set $\mathcal W$ of all   $3$-manifolds with  $RP^2$-singularities and nonempty boundary.
\end{theorem}

 If the boundary curves of all $2$-components of a special polyhedron $P$ have trivial normal bundles, then, by construction, $W(P)$ is a genuine $3$-manifold without singularities and, moreover,   $P$ is embedded into   $W(P)$ so that $W(P)\setminus P$ is homeomorphic to $\partial W(P)\times (0,1]$. In this case, $P$ is a spine of $W(P)$ in the   following sense.

\begin{definition}
 Let $M$ be a connected compact $3$-manifold. A compact two-dimensional polyhedron $P\subset M$ is called a spine  of $M$ if either $\partial M \neq \emptyset$ and $M\setminus P$ is homeomorphic to $\partial M \times (0,1]$ or $\partial M = \emptyset$ and $M \setminus P$ is homeomorphic to an open ball.
\end{definition}

 A spine of a disconnected $3$-manifold is the union of spines of its connected components.

A spine of a $3$-manifold is said to be \emph{special} if it is a special
polyhedron.

\begin{theorem} \cite{Casler}
 \label{thm:Casler}
Any compact $3$-manifold has a special spine.
\end{theorem}

 The role of the moves $T$ and $T^{- 1}$ is clear from the following theorem.

\begin{theorem} \cite[Theorem 1.2.5]{MatBook}
 \label{thm:equivalence}
 Let $P_1$ and $P_2$ be special polyhedra having at least two true vertices each. Then the following holds:
\begin{enumerate}
 \item If $P_1$ and $P_2$ are special spines of the same $3$-manifold, then $P_1$ can be transformed into $P_2$ by a finite sequence of moves $T^{\pm1}$.
 \item If $P_1$ can be transformed into $P_2$ by a finite sequence of moves $T^{\pm1}$ and if $P_1$ is a special spine of a $3$-manifold, then $P_2$ is a special spine of the same $3$-manifold.
\end{enumerate}
\end{theorem}

\begin{remark} \label{remark:[P]_of_spines}
 Let $P$ be a special spine of a $3$-manifold $M$. Suppose that $P$ has   at least two true vertices. Theorem~\ref{thm:equivalence} implies that the equivalence class $[P]$ of $P$ consists precisely of the special spines of  $M$ with at least two true vertices.
\end{remark}

\begin{remark}
 Consider the restriction of the map $\varphi : \mathcal V \to \mathcal W$ from Theorem \ref{thm:surjection} to the set of classes of special spines with at least two   true vertices.   Theorems \ref{thm:Casler} and \ref{thm:equivalence} imply that this restriction is a bijection on the set of   compact 3-manifolds with nonempty boundary.
\end{remark}

 Turaev-Viro invariants~\cite{TV} play an important role in the theory of three-dimensional manifolds. The following theorem justifies their use   in the study of   virtual manifolds.

\begin{theorem} \cite[Theorem 4]{Mat2009}
 \label{thm:TVextension}
 Turaev-Viro invariants of 3-manifolds   extend  to the class of
  virtual 3-manifolds.
\end{theorem}

\section{Complexity of virtual 3-manifolds}
 \label{section3}

\begin{definition}
 The  complexity  $\operatorname{cv}[P]$ of a virtual 3-manifold $[P]$ is equal to $k$ if the equivalence class $[P]$ contains a special polyhedron with $k$ true vertices and contains no special polyhedra with a smaller number of true vertices.
\end{definition}

 We establish some properties of the complexity $\operatorname{cv}[P]$.

\begin{lemma}
 \label{lemma:elementary_properties}
 Let $P$ be a special polyhedron. Then the following statements hold.
\begin{enumerate}
 \item $\operatorname{cv}[P] \geq 1$.
 \item The following conditions are equivalent:
 \begin{enumerate}
  \item $\operatorname{cv}[P] = 1$.
  \item $P$ has exactly one true vertex.
  \item $[P]=\{P\}$.
 \end{enumerate}
 \item If $P$ has exactly two true vertices then $\operatorname{cv}[P] = 2$.
\end{enumerate}
\end{lemma}

\begin{proof}
 1. The claim follows from the definition, since each special polyhedron has at least one true vertex.

 2.  The claim is obvious, since each one-vertex special polyhedron is the only element in its equivalence class.

 3. If $P$ has exactly two true vertices, then $\operatorname{cv}[P] \leqslant 2$. The second claim of the lemma implies that $\operatorname{cv}[P] \neq 1$. This completes the proof.
\end{proof}

 Note the following relationship between the number of $2$-components of a special polyhedron and the number of boundary components of the corresponding manifold with $RP^{2}$-singularities.

\begin{lemma}
 \label{lemma:d>=b}
 Let $P$ be a special polyhedron with $d$ $2$-components. Denote by $b$ the number of boundary components of the $3$-manifold  with $RP^2$-singularities $W(P)$. Then   $d \geq b$.
\end{lemma}

 By the above construction, the $3$-manifold   with $RP^2$-singularities $W(P)$ has a non-empty boundary, i.e. $b \geq 1$. The proof of Lemma~\ref{lemma:d>=b} is similar to that of Lemma 2.1 in~\cite{FrigMarPet03}.

 Using Lemma~\ref{lemma:d>=b} we establish two-sided estimates of the complexity.

\begin{theorem}
 \label{thm:complexity_bounds}
Let $P$ be a special polyhedron with $n \geq 2$ true vertices and $d$ $2$-components. Denote by $b$ the number of boundary components of the $3$-manifold  with $RP^2$-singularities $W(P)$. Then
$$
n - (d-b) \leq \operatorname{cv}[P] \leq n.
$$
\end{theorem}

\begin{proof}
 The definition of the complexity implies the right inequality $\operatorname{cv}[P] \leq n$. Let us prove the left inequality. Since the special graph of any special polyhedron is   4-regular,   $\chi(P) = n - 2n + d = d - n$. In the class $[P]$ we choose a special polyhedron with $\operatorname{cv}[P]$ true vertices and denote it by $P '$. Let $m$ be the number of $2$ -components of the polyhedron $P'$. Then   $\chi(P') = m - \operatorname{cv}[P]$. The equality $\chi(P') = \chi(P)$ implies that $\operatorname{cv}[P] = n - d + m$. Since the   map $\varphi : \mathcal V \to \mathcal W$ considered above is well defined, the manifolds $W(P')$ and $W(P)$ are homeomorphic. Applying Lemma~\ref{lemma:d>=b} to the polyhedron $P '$, we get $m \geq b$. Therefore, $\operatorname{cv}[P] \geq n - d + b$, which completes the proof of the theorem.
\end{proof}

\begin{corollary} \label{cor1}
 If, under the assumptions of Theorem~\ref{thm:complexity_bounds}, we have   $d = b$, then $\operatorname{cv} [P] = n$. In particular,  if the special polyhedron $P$ has exactly one $2$-component, then $\operatorname{cv} [P] = n$.
\end{corollary}

 The following theorem allows us to calculate the complexity of virtual manifolds presented by special polyhedra with two $2$-components. Recall that the boundary curve of a $2$-component   is said to be short if it passes through three true vertices   and through each of them exactly once.

\begin{theorem}
 \label{thm:main}
 Let $P$ be a special polyhedron having $n \geq 2$ true vertices and two $2$-components whose boundary curves are not short. Then   $\operatorname{cv} [P] = n$.
\end{theorem}

 The proof of Theorem~\ref{thm:main} occupies Section~\ref{section4} below.

\begin{lemma}
 \label{lemma:complexity_properties}
 Let $P$ be a special polyhedron with $n \geq 2$ true vertices and two $2$-components. Suppose that the boundary curve of at least one of the $2$-components is short. Then:
\begin{enumerate}
 \item If the short boundary curve has a trivial normal bundle, then $\operatorname{cv}[P] = n-1$.
 \item If both boundary curves have non-trivial normal bundles, then $\operatorname{cv}[P] = n$.
\end{enumerate}
\end{lemma}

\begin{proof}
 1. By Remark \ref{remark:adm_T^{-1}},   we can apply  to~$P$ the move $T^{-1}$. This gives a special polyhedron~$P'$  equivalent to $P$. Therefore, $\operatorname{cv} [P] = \operatorname{cv} [P']$. By the construction, the polyhedron $P'$ has $(n-1)$ true vertices and one $2$-component. By Corollary~\ref{cor1}, we have $\operatorname{cv} [P'] = n-1$.

 2.  Since $P$ has two $2$-components, Lemma~\ref{lemma:d>=b} implies that the number of   boundary components of the manifold $W(P)$ is   either $1$ or $2$. Then by Theorem~\ref{thm:complexity_bounds}, we have  $n-1 \leq \operatorname{cv} [P] \leq n$. We check now that $\operatorname{cv} [P] \neq n-1$.

Suppose, on the contrary, that $\operatorname{cv} [P] = n-1$, i.e. that $P$ is equivalent to a special polyhedron $P'$ with $n-1$ true vertices. Since all polyhedra in the same equivalence class have the same Euler characteristic,   $\chi(P') = \chi(P) = 2 - n$. Therefore, the polyhedron $P'$ has exactly one $2$-component.

If both boundary curves of $P$ have non-trivial normal bundles, then   Lemma \ref{lemma:number_of_nontrivial} implies that $P'$ has  two $2$-components whose boundary curves have non-trivial normal bundles. This contradicts the fact that $P'$ has only one $2$-component. This completes the proof.
\end{proof}

\begin{remark}
 Let $P$ be a special polyhedron with $n \geq 2$ true vertices and two $2$-components. Using simple combinatorial arguments it is easy to show that both boundary curves of the $2$-components of $P$ cannot be short.
\end{remark}

\section{Proof of theorem~\ref{thm:main}} \label{section4}
 \label{section4}

 First,  note that if $n=2$ then $\operatorname{cv}[P]=2$ by Claim (3) of Lemma \ref{lemma:elementary_properties}. Therefore we will assume that $ n\geq 3$.

An argument analogous to the one in the proof of  Lemma \ref{lemma:complexity_properties}, Claim (2)   proves the inequalities $n-1 \leq \operatorname{cv} [P] \leq n$. It remains to show that $\operatorname{cv} [P] \neq n-1$.

 Suppose, on the contrary,  that $\operatorname{cv} [P] = n-1$, i.e., that in the class $[P]$ there is a special polyhedron $P'$ with $n-1$ true vertices and, hence, with one $2$-component.

 A simple polyhedron contained in a special polyhedron is   \emph{proper}, if it is different from the empty set and from the whole polyhedron. The following lemma describes all proper simple subpolyhedra of~$P$.

\begin{lemma} \label{lemma41}
 The polyhedron $P$ has exactly one proper simple subpolyhedron.
\end{lemma}

\begin{proof}
 As mentioned above, the Turaev--Viro invariants  extend to
  virtual manifolds. To prove the lemma, we use one of the resulting invariants, known as the $\varepsilon $-invariant, which is the homologically trivial part of the order $5$ Turaev--Viro invariant. We give the definition of the $\varepsilon $-invariant following~\cite{MatBook}. Consider a special polyhedron $R$ and let $\mathcal{F}(R)$ be the set of all   simple subpolyhedra of $R$ including $R$ and the empty set. Set $\varepsilon = (1+\sqrt{5})/2$,  a solution of the equation $\varepsilon^2=\varepsilon+1$. We associate to each $Q\in \mathcal{F}(R)$ its \emph{$\varepsilon$-weight} by the formula
$$
w_{\varepsilon}(Q) = (-1)^{V(Q)}\varepsilon^{\chi(Q)-V(Q)},
$$
where $V(Q)$ is the number of   true vertices of the  simple polyhedron  $Q$ and $\chi $ is the Euler characteristic. Set
$$
t(R) = \sum_{Q\in \mathcal{F} (R)} w_{\varepsilon}(Q).
$$
As shown in~\cite{MatBook}, the number $t(R)$ is invariant under the moves $T^{\pm1}$.
We   call $t(R)$ the $\varepsilon$-invariant of the virtual manifold $[R]$ and denote it by $t[R]$.

We return to the proof of the lemma and use the $\varepsilon$-invariant to show that $P$ has at least one proper simple subpolyhedron. If it is not the case, then $$t(P) = (-1)^{n} \varepsilon^{2-2n} + 1 \quad {\text {whereas}} \quad
    t(P') = (-1)^{n-1} \varepsilon^{3- 2n}  + 1 .
$$
Since $\varepsilon = (1+\sqrt{5})/2$, the values $t(P)$ and $t(P')$ differ. This contradicts the fact that $P'\in [P]$. Thus, $P$ has at least one proper simple polyhedron.

 We will show that such a subpolyhedron is unique. By the compactness of a simple subpolyhedron if it contains a  point of a $2$-component, then it contains the $2$-component entirely. Thus, to describe a simple subpolyhedron of $P$ it is enough to indicate which
$2$-components of $P$ it includes (its triple points
and   true vertices will be then determined uniquely). Since~$P$ has only two $2$-components,   it has at most two proper simple subpolyhedra, and each of them contains exactly one $2$-component. Since   $P$ is connected, it has an edge traversed twice by the boundary curve of one of the $2$-components  and traversed once by the boundary curve   of the other $2$-component. The latter $2$-component cannot be contained in a proper simple subpolyhedron of $P$. This completes the proof.
\end{proof}

 Denote by $Q$ the proper simple subpolyhedron of $P$ whose existence was proved in Lemma~\ref{lemma41}. Let $V(Q)$ be the number of true vertices of $Q$.

\begin{lemma} \label{lemma42}
We have $V(Q) = n-3$.
\end{lemma}

\begin{proof}
 Since $Q$ is a proper simple subpolyhedron of $P$, we have $V(Q)<n$. Clearly,
$$
t[P] = t(P) = (-1)^{n} \varepsilon^{2- 2n} + (-1)^{V(Q)} \varepsilon^{\chi(Q) - V(Q)} + 1.
$$
As in the proof of Lemma~\ref{lemma41}, the calculation of the same $\varepsilon$-invariant $t[P]$   on the polyhedron $P'$ gives $t[P] = t(P') = (-1)^{n-1} \varepsilon^{3- 2n}  + 1$. It is easy to see that $t(P) = t(P')$ if and only if
$$
(-1)^{V(Q)} \varepsilon^{\chi(Q) - V(Q)} = (-1)^{n-1} \left[ \varepsilon^{3-2n} + \varepsilon^{2-2n} \right].
$$
This equality holds if and only if the following conditions hold:
\begin{itemize}
\item[(i)] $V(Q)$ and $n-1$ have the same parity; \\
\item[(ii)] $\varepsilon^{\chi(Q) - V(Q)} = \varepsilon^{3-2n} + \varepsilon^{2-2n}$.
\end{itemize}

Since
\begin{equation*} \label{eqn200}
\varepsilon^{3-2n} + \varepsilon^{2-2n} = \varepsilon^{2-2n}  (\varepsilon + 1) = \varepsilon^{2-2n} \varepsilon^{2} = \varepsilon^{4-2n},
\end{equation*}
condition (ii) is equivalent to
\begin{equation} \label{eqn1}
\chi(Q) + 2n = 4 + V(Q).
\end{equation}

 Now we show that $V(Q) \neq n-1$. Denote by $\alpha$ and $\beta$ the $2$-components of $P$ assuming for concreteness that $\alpha$ is contained in $Q$, and $\beta$ is not. Suppose, on the contrary, that $V(Q) = n-1$. Then the boundary curve $\partial \beta$ passes through only one true vertex of  $P$, and hence through one edge of $P$ (see  Figure~\ref{fig2}).
\begin{figure}[h]
\begin{center}
\includegraphics[scale=0.7]{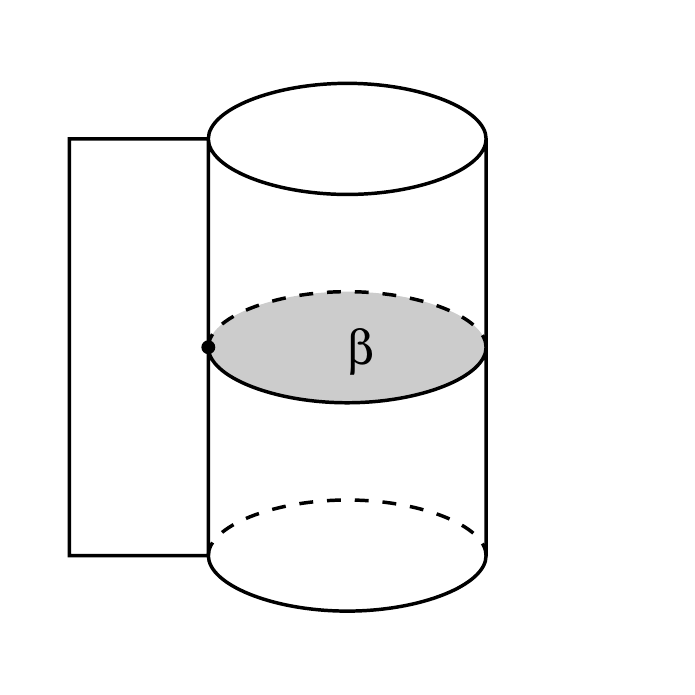}
\end{center}
\caption{The boundary curve $\partial \beta$ passes through one true vertex and through one edge of the special polyhedron} \label{fig2}
\end{figure}
Hence, $Q = P \setminus \beta$. Therefore $\chi(Q) = \chi(P) - 1= 1-n$. Obviously, for such values of $V(Q)$ and $\chi(Q)$ the equality (\ref{eqn1}) is not met. This contradiction implies that $V(Q) \neq n-1$.

 Recall that the special polyhedron $P$ has a natural cell decomposition: a $2$-cell is a connected component of the set of nonsingular
points, a $1$-cell is an edge, and a $0$-cell is a true vertex. The cell decomposition of $P$ induces a cell decomposition of its simple subpolyhedron~$Q$. Denote by $k_{0}$ the number of $0$-cells of $Q$  and by $k_{1}$ the number of  $1$-cells of $Q$. Hence, $k_{0} \leq n$ and $k_{1} \leq 2 k_{0}$. Since $Q$ has exactly one $2$-cell (that is the $2$-component $\alpha$),   $$\chi(Q) = k_{0} - k_{1} + 1 \geq 1 - k_{0 }\geq 1 - n.$$ Hence, using the formula (\ref{eqn1}), we obtain that $V(Q) \geq n - 3$. Taking into account that $V(Q)$ has the same parity as $n-1$  and that $V(Q) \neq n-1$, we can conclude that $V(Q)= n-3$. The proof is complete.
\end{proof}

 Recall that by the choice of notation  in the proof of Lemma~\ref{lemma42}, the $2$-component $\beta$ of the polyhedron $P$ is not contained in the simple subpolyhedron $Q$.

\begin{lemma} \label{lemma:boundary_curve}
 The boundary curve $\partial \beta$ of the $2$-component $\beta$ of the polyhedron $P$ has the following properties:
\begin{itemize}
\item[(a)] $\partial \beta$ passes along each edge of $P$ at most once;
\item[(b)] $\partial \beta$ passes through three true vertices of $P$ and through each of them at most twice.
\end{itemize}
\end{lemma}

\begin{proof}
 We show that the simple subpolyhedron $Q\subset P$ contains all true vertices of $P$ (they are not necessarily true vertices in $Q$) and all its edges. That is, in the above notations, we have $k_{0} = n$ and $k_{1} = 2n$. Indeed, if $k_{0} < n$ then $k_{1} < 2n-3$. Hence, $\chi(Q)  = k_{0} - k_{1} +1 > k_{0} - 2n + 4$. Since $V(Q)= n-3$, by the formula (\ref{eqn1}) we have $\chi(Q)  = 1-n$. Therefore, $k_{0} < n-3$, which contradicts the condition $k_{0} \geq V(Q) = n-3$. Thus $k_{0} = n$. Then the equalities $k_{0} - k_{1} + 1 = 1 - n$ and $k_{0}= n$ imply that $k_{1} = 2n$.

Since the   polyhedron $Q$ contains all edges of $P$, the curve $\partial \beta$ passes along each edge of $P$ at most once and, therefore, it passes through each true vertex of $P$ at most twice. Since $V(Q) = n-3$, the curve $\partial \beta$ passes through exactly three true vertices of $P$. The proof is complete.
\end{proof}

 Let us apply a move $T$ to the polyhedron $P'$ in an arbitrary way. Denote the resulting special polyhedron by $P''$. The polyhedron $P''$ has $n$ true vertices, two $2$-components and belongs to the equivalence class~$[P]$.

\begin{lemma} \label{lemma44}
 The virtual $3$-manifolds $[P]$ and $[P'']$ have different Turaev--Viro invariants.
\end{lemma}

\begin{proof}
 We will show that these virtual manifolds differ by the order 7 Turaev--Viro invariant.  We recall the construction of Turaev--Viro invariants \cite{TV} following the exposition in \cite{MatBook}  and calculate  the value of the order 7 invariant for $[P]$ and $[P'']$.

 Let $R$ be a special polyhedron, $V(R)$ be the set of its vertices, and $U(R)$ be the set of its $2$-components. A \emph{coloring} of $R$ in the color palette $\mathcal C = \{ 0, 1, \ldots, r-2 \}$, where $r\geq 3$ is an integer, is a map $\zeta : U(R) \to \mathcal C$. An unordered triple $i, j, k$ of colors  from the palette~$\mathcal C$ is   \emph{admissible} if
\begin{itemize}
\item{} $i+j \geq k$, $j + k \geq i$, $k + i \geq j$;
\item{} $i + j + k$ is even;
\item{} $i+j+k \leq 2r - 4$.
\end{itemize}
 A coloring $\zeta$ of a special polyhedron $R$ is   \emph{admissible} if the colors of any three $2$-components adjacent to the same edge form an admissible triple. The set of all admissible colorings of $R$ in the palette $\mathcal C$ will be denoted by $\text{\rm Col} (R)$.

 We describe all admissible colorings of the special polyhedra $P$ and $P''$ when $r=7$. It was established in Lemma \ref{lemma:boundary_curve}, that the boundary curve $\partial \beta$ passes through exactly three true vertices of $P$. Denote by $m$ the number of those vertices traversed by the curve $\partial \beta$ exactly twice. By the conditions of Theorem \ref{thm:main}, the boundary curves of the $2$-components of $P$ are not short, and therefore the   curve $\partial \beta$ traverses twice at least one     true vertex of $P$. Thus, $1 \leq m \leq 3$. As in the proof of Lemma~\ref{lemma42}, we let $\alpha $  denote the second $2$-component of  $P$. We separate all edges of $P$  into two types. An edge is of type I, if only the curve $\partial \alpha$ passes along this edge. An edge is of type II, if the curve $\partial \alpha$ passes along it twice and the curve $\partial \beta$ passes along it once.

 Recall that the special polyhedron $P''$ is obtained by applying the move $T$ to the special polyhedron $P'$. Therefore the polyhedron $P''$, as well as $P$, has two $2$-components, and the boundary curve of one of its $2$-components is short. Denote this $2$-component of  $P''$ by $\delta$, and denote the second $2$-component of  $P''$ by $\gamma$. Similarly to $P$, all edges of $P''$  are separated into two types. An edge is of type I, if only the curve $\partial \gamma$ goes along this edge, and of type II, if the curve $\partial \gamma$ goes along   it twice and the curve $\partial \delta$ goes along it once.

 Let $(k,\ell)$ denote a coloring of $P$ such that the $2$-component $\alpha$ is painted in the color $k$ and the $2$-component $\beta$ in the color $\ell$. Analogously, let $(k,\ell)$ denote a coloring of $P''$ such that the $2$-component $\gamma$ is painted in the color $k$ and the $2$-component $\delta$ in the color $\ell$. A coloring $(k,\ell)$ is called \emph{monochrome} if $k=\ell$. Consider the palette ${\mathcal C} = \{0, 1, 2, 3, 4, 5\}$ for $r=7$. From the definition of an admissible coloring it  follows immediately that each of the polyhedra $P$ and $P''$ has four admissible colorings. More precisely,
$$\text{\rm Col} (P) = \text{\rm Col} (P'')  = \{ (0,0), (2,2), (2,0), (2,4) \}.$$

 Recall relevant notation from the theory of quantum invariants. Let $q$ be a $2r$-th root of unity such that $q^{2}$ is a primitive root of unity of degree $r$. For a non-negative integer $n$ set
 $$[n] = \frac{q^{n}- q^{-n}}{q - q^{-1}} \qquad \text{ and } \qquad [n]! = [n] [n-1] \cdots [2] [1].$$
In particular, $[1]! = [1] = 1$. By definition, $[0]! = 1$.

To each color $i \in \mathcal C$ one assigns a weight
 $$w_{i} = (-1)^{i} [i+1].$$
In particular,
\begin{equation} \label{w_i}
w_{0} = 1, \qquad w_{2} = [3], \qquad w_{4}=[5].
\end{equation}

 For a non-negative integer $m$ we write $\hat{m} = m / 2$. For an admissible triple $i,j,k$ put
$$
\Delta(i,j,k) = \left(  \frac{[\hat{i} + \hat{j} - \hat{k}]! \,  [ \hat{j}  + \hat{k} - \hat{i}]! \,  [ \hat{k} + \hat{i} - \hat{j} ] ! }{  [  \hat{i} + \hat{j} + \hat{k} + 1]! }  \right)^{1/2} .
$$
In particular,
$$
\Delta(0,0,0) = 1, \quad \Delta(2,2,2) = \frac{1}{([4]!)^{1/2}}, \quad
\Delta(2,2,0) = \frac{1}{[3]^{1/2}}, \quad \Delta(2,2,4) = \frac{[2]!}{([5!])^{1/2}}.
$$

     Consider a regular neighborhood $N(v)$ of a true vertex $v$ of a special polyhedron $R$. The intersection of $N(v)$ with the union of   $2$-components of $R$ consists of six discs which, as in \cite{MatBook},   are called \emph{wings} of   $N(v)$. Two wings are \emph{opposite} if   their closures in $N(v)$ meet only in~$v$. It is clear that the   wings of $N(v)$ are divided into three pairs of opposite wings. Note that any coloring of $R$ induces a coloring of the wings of   $N(v)$. If the coloring of $R$ is admissible and triple pairwise opposite wings are painted in the colors $i$ and $l$, $j$ and $m$, $k$ and $n$ (see Figure~\ref{fig: colored_butterfly}), then with the vertex $v$ one associates a   quantum $6j$-symbol $\left| \begin{matrix} i & j & k \\ l & m & n \end{matrix} \right|_{v} \in \mathbb C$.

\begin{figure}[h]
\begin{center}
\includegraphics[scale=1.2]{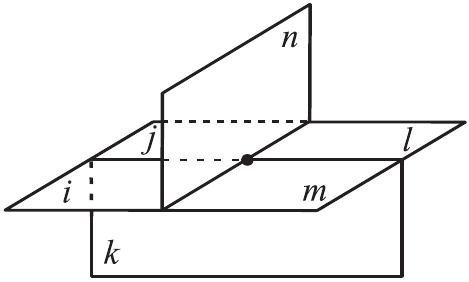}
\end{center}
\caption{A coloring of the wings of a true vertex neighborhood.}
\label{fig: colored_butterfly}
\end{figure}

The $6j$-symbol is defined by the formula
$$
\left| \begin{matrix} i & j & k \\ l & m & n \end{matrix} \right|_{v} = \sum_{z} \frac{ (-1)^{z} \, [z+1]! \, A \left( \begin{matrix} i & j & k \cr l  & m & n \end{matrix} \right) }{ B \left(z,  \begin{pmatrix} i & j & k \cr l & m & n \end{pmatrix} \right) \, C \left( z, \begin{pmatrix} i & j & k \cr l & m & n \end{pmatrix} \right) },
$$
where
$$
A \left( \begin{matrix} i & j & k \cr l & m & n \end{matrix} \right) = (\sqrt{-1})^{(i+j+k+l+m+n)} \, \Delta(i,j,k) \, \Delta(i,m,n) \, \Delta(j,l,n) \, \Delta(k,l,m),
$$
$$
B \left( z, \left( \begin{matrix} i & j & k \cr l & m & n \end{matrix} \right) \right) = [ z - \hat{i} - \hat{j} - \hat{k} ]! \, [z- \hat{i} - \hat{m} - \hat{n} ]! \, [z-\hat{j} - \hat{l} - \hat{n} ]! \, [ z - \hat{k} - \hat{l} - \hat{m} ]! ,
$$
$$
C \left( z, \left( \begin{matrix} i & j & k \cr l & m & n \end{matrix} \right) \right) = [ \hat{i} + \hat{j} + \hat{l} + \hat{m} - z ]! \, [ \hat{i} + \hat{k} + \hat{l} + \hat{n} - z]! \, [ \hat{j} + \hat{k} + \hat{m} +  \hat{n} - z]! ,
$$
and the sum is taken over all integer $z$ such that $a \leq z \leq b$, where
$$
a = \max \{ \hat{i} + \hat{j} + \hat{k}, \hat{i} + \hat{m} + \hat{n}, \hat{j} + \hat{l} + \hat{n}, \hat{k} + \hat{l} + \hat{m} \},
$$
$$
b = \min \{ \hat{i} + \hat{j} + \hat{l} + \hat{m}, \hat{i} + \hat{k} + \hat{l} + \hat{n}, \hat{j} + \hat{k} + \hat{m} + \hat{n} \}.
$$
By these formulas, for admissible colorings of $2$-components of $P$ and $P''$, we obtain:
$$
\begin{gathered}
A \left( \begin{matrix} 2 & 2& 0 \cr 2 & 2 & 0 \end{matrix} \right) = \frac{1}{[3]^{2}}, \quad
B \left( 2,  \left( \begin{matrix} 2 & 2& 0 \cr 2 & 2 & 0 \end{matrix} \right) \right) = 1, \quad
C \left( 2, \left( \begin{matrix} 2 & 2& 0 \cr 2 & 2 & 0 \end{matrix} \right) \right) = [2]!, \\
A \left( \begin{matrix} 2 & 2& 0 \cr 2 & 2 & 2 \end{matrix} \right) = \frac{-1}{[3] [4]!}, \quad
B \left( 3, \left( \begin{matrix} 2 & 2& 0 \cr 2 & 2 & 2 \end{matrix} \right) \right) = 1, \quad
C \left( 3, \left( \begin{matrix} 2 & 2& 0 \cr 2 & 2 & 2 \end{matrix} \right) \right) = 1, \\
A \left( \begin{matrix} 2 & 2& 0 \cr 2 & 2 & 4 \end{matrix} \right) = \frac{([2]!)^{2}}{[3] [5]!}, \quad
B \left( 4, \left( \begin{matrix} 2 & 2& 0 \cr 2 & 2 & 4 \end{matrix} \right) \right) = ([2]!)^{2}, \quad
C \left( 4, \left( \begin{matrix} 2 & 2& 0 \cr 2 & 2 & 4 \end{matrix} \right) \right) = 1, \\
A \left( \begin{matrix} 2 & 2& 2 \cr 2 & 2 & 2 \end{matrix} \right) = \frac{1}{([4]!)^{2}}, \quad
B \left( 3, \left( \begin{matrix} 2 & 2& 2 \cr 2 & 2 & 2 \end{matrix} \right) \right) = 1, \quad
C \left( 3, \left( \begin{matrix} 2 & 2& 2 \cr 2 & 2 & 2 \end{matrix} \right) \right) = 1, \\
B \left( 4, \left( \begin{matrix} 2 & 2& 2 \cr 2 & 2 & 2 \end{matrix} \right) \right) = 1, \quad
C \left( 4, \left( \begin{matrix} 2 & 2& 2 \cr 2 & 2 & 2 \end{matrix} \right) \right) = 1,
\end{gathered}
$$
$$
\begin{gathered}
A \left( \begin{matrix} 2 & 2& 2 \cr 2 & 2 & 4 \end{matrix} \right) = \frac{- ([2]!)^{2}}{[4]! [5]!}, \quad
B \left(4,  \left( \begin{matrix} 2 & 2& 2 \cr 2 & 2 & 4 \end{matrix} \right) \right) = 1, \quad
C \left(4,  \left( \begin{matrix} 2 & 2& 2 \cr 2 & 2 & 4 \end{matrix} \right) \right) = 1, \\
A \left( \begin{matrix} 2 & 2& 4 \cr 2 & 2 & 4 \end{matrix} \right) = \frac{([2]!)^{4}}{([5]!)^{2}}, \quad
B \left( 4, \left( \begin{matrix} 2 & 2& 4 \cr 2 & 2 & 4 \end{matrix} \right) \right) = 1, \quad
C \left( 4, \left( \begin{matrix} 2 & 2& 4 \cr 2 & 2 & 4 \end{matrix} \right) \right) = 1.
\end{gathered}
$$
Thus,
\begin{equation} \label{6j}
\begin{gathered}
\left| \begin{matrix} 0 & 0 & 0 \\ 0 & 0 & 0 \end{matrix} \right| = 1, \quad
\left| \begin{matrix} 2 & 2 & 0 \\ 2 & 2 & 0 \end{matrix} \right| = \frac{1}{[3]}, \quad
\left| \begin{matrix} 2 & 2 & 0 \\ 2 & 2 & 2 \end{matrix} \right| = \frac{1}{[3]}, \quad  \\
\left| \begin{matrix} 2 & 2 & 2 \\ 2 & 2 & 2 \end{matrix} \right| = \frac{[5] - 1}{[4]!},  \quad
\left| \begin{matrix} 2 & 2 & 2 \\ 2 & 2 & 4 \end{matrix} \right| = - \frac{([2]!)^{2}}{[4]!}, \quad
\left| \begin{matrix} 2 & 2 & 4 \\ 2 & 2 & 4 \end{matrix} \right| = \frac{([2]!)^{2}}{[5]!}.
\end{gathered}
\end{equation}

In general, the weight of an admissible coloring $\zeta \in \text{\rm Col} (R)$ is defined by the rule
$$
w (\zeta, R) = \Pi_{v \in V(R)} \left| \begin{matrix} i & j & k \\ l & m & n \end{matrix} \right|_{v} \, \Pi_{u \in U(R)} w_{\zeta(u)}.
$$
The  \emph{weight of the special polyhedron} $R$ is defined as the sum of the weights of all its admissible colorings:
$$
w(R) = \sum_{\zeta \in \text{\rm Col} (R)} w(\zeta, R).
$$
By Theorem \ref{thm:TVextension} the weight $w(R)$ of $R$ is the \emph{order $r$ Turaev--Viro invariant} of the virtual manifold $[R]$.

Now consider again our polyhedra $P$ and $P''$. By definition, the weight of   $P$ is the sum of the weights of all admissible colorings:
$$
w(P) = w((0,0), P) + w((2,0), P) + w((2,2), P) + w((2,4), P) ,
$$
and similarly,
$$
w(P'') = w((0,0), P'') + w((2,0), P'') + w((2,2), P'') + w((2,4), P'').
$$
Since the special polyhedra $P$ and $P''$ have the same number of true vertices, the weights of monochrome colorings obviously coincide:
$$w((0,0), P) = w((0,0), P'')  \qquad \text{ and } \qquad w((2,2), P) = w((2,2), P'').$$

Next we compare the weights of the coloring $(2.0)$. Recall that $m\in \{1,2,3\}$ is the number of true vertices of $P$ traversed by the curve $\partial \beta$ exactly twice.   Taking into account the formula (\ref{w_i}), we obtain that
$$
w((2,0), P) =
\left| \begin{matrix} 2 & 2 & 2  \\ 2 & 2 & 2 \end{matrix} \right|^{n-3}
\left| \begin{matrix} 2 & 2 & 0  \\ 2 & 2 & 0 \end{matrix} \right|^m
\left| \begin{matrix} 2 & 2 & 2  \\ 2 & 2 & 0 \end{matrix} \right|^{3-m} [3],
$$
$$
w((2,0), P'') =
\left| \begin{matrix} 2 & 2 & 2  \\ 2 & 2 & 2   \end{matrix} \right|^{n-3}
\left| \begin{matrix} 2 & 2 & 2  \\ 2 & 2 & 0  \end{matrix} \right|^3  [3].
$$
It follows from (\ref{6j}) that
 $$w((2,0), P) = w((2,0), P'').$$
Let us compare the weights of the coloring $(2,4)$:
$$
w((2,4), P) =
\left| \begin{matrix} 2 & 2 & 2  \\ 2 & 2 & 2  \end{matrix} \right|^{n-3}
\left| \begin{matrix} 2 & 2 & 4  \\ 2 & 2 & 4 \end{matrix} \right|^m
\left| \begin{matrix} 2 & 2 & 2  \\ 2 & 2 & 4   \end{matrix} \right|^{3-m}  [3] [5],
$$
$$
w((2,4), P'') =
\left| \begin{matrix} 2 & 2 & 2  \\ 2 & 2 & 2  \end{matrix} \right|^{n-3}
\left| \begin{matrix} 2 & 2 & 2  \\ 2 & 2 & 4  \end{matrix} \right|^3 [3] [5].
$$
Using the formulas (\ref{6j}) and the condition $m \geqslant 1$ we get
$$w((2,4), P) \neq w((2,4), P''),$$
 and hence
$$w(P) \neq w(P'').$$
Thus, the virtual manifolds $[P]$ and $[P'']$ differ by the order $7$ Turaev--Viro invariant. This completes the proof.
\end{proof}

We  proved that the equivalent polyhedra $P$ and $P''$, and hence,   $P$ and $P'$,  have different Turaev--Viro invariants. This   contradiction shows that the original assumption on the existence of a special polyhedron $P'\in [P]$ with $n-1$ true vertices is false. This completes the proof of Theorem~\ref{thm:main}.

\section{Complexity of manifolds}
 \label{section5}

 In this section we   discuss the relationships between the  results above on the complexity of virtual $3$-manifolds with the complexity theory of classical $3$-manifolds.

 A compact polyhedron $P$ is said to be \emph{almost simple} if the link of each of its points can be embedded into the complete graph $K_4$ with four vertices. The points whose links are homeomorphic to   $K_4$ are the true vertices of~$P$. A spine of a $3$-manifold is said to be \emph{almost simple} if it is an almost simple polyhedron.

\begin{definition} \label{def:complexity}
 The  complexity  $c(M)$ of a compact $3$-manifold $M$ is the non-negative integer~$n$ such that $M$ has an almost simple spine with $n$ true vertices and has no almost simple spines with fewer true vertices.
\end{definition}

 The following theorem shows that the complexity of a $3$-manifold can be estimated from above by the complexity of the corresponding virtual manifold.

\begin{theorem} \label{theorem4}
 Let $P$ be a special spine of a $3$-manifold $M$. Then   $\operatorname{c}(M) \leq \operatorname{cv}[P]$.
\end{theorem}

\begin{proof}
 By Definition \ref{def:complexity}, the complexity of $M$ does not exceed the number of   true vertices of any almost simple spine of $M$. Consequently, $\operatorname{c}(M)$ does not exceed the number of true vertices of any special spine of~$M$ as all special polyhedra are almost simple. Therefore, $\operatorname{c}(M) \leq \operatorname{cv}[P]$.
\end{proof}

\begin{remark} \label{remark2}
   The inequality of Theorem~\ref{theorem4} is strict for some compact $3$-manifolds. For example, if $P$ is a special spine of a manifold $M$ of complexity $0$, then $\operatorname{c}(M) < \operatorname{cv}[P]$ because by Lemma \ref{lemma:elementary_properties} we have $\operatorname{cv} [P] \geq 1$. Closed manifolds of complexity $0$ are $S^{3}$, $\mathbb RP^{3}$, and the lens space $L_{3,1}$. Orientable irreducible and boundary irreducible $3$-manifolds with nonempty boundary and complexity~$0$ are described in \cite{MatNik14}.
\end{remark}

 We describe a class of $3$-manifolds, for which the inequality in Theorem ~\ref{theorem4} becomes the equality.

\begin{theorem} \label{theorem5}
 Let $P$ be a special spine of an irreducible boundary irreducible $3$-manifold $M$ such that $\operatorname{c}(M) \neq 0,1$ and all proper annuli in $M$ are inessential. Then   $\operatorname{c}(M) = \operatorname{cv} [P]$.
\end{theorem}

\begin{proof}
 Let $\widehat{P}$ be a minimal almost simple spine of $M$, i.e. an almost simple spine with $\operatorname{c}(M)$ true vertices. Since $M$ is an irreducible and boundary irreducible and since all proper annuli in $M$ are inessential,   \cite[Theorem 2.2.4] {MatBook} implies that there is a special spine $P_1$ of $M$ having the same number $\operatorname{c}(M)$ of true vertices as $\widehat{P}$. The special spines $P_1$ and $P$ have at least two true vertices each, because $\operatorname{c}(M) \geq 2$. Therefore, by Remark \ref{remark:[P]_of_spines}, we have $[P_1] = [P]$ and $\operatorname{cv}[P] = \operatorname{cv} [P_{1}] = c(M)$.
\end{proof}

 It is known that all hyperbolic $3$-manifolds are irreducible, have incompressible boundary, and contain  no essential annuli. Therefore, Theorem~\ref{theorem5} implies the following corollary.

\begin{corollary} \label{cor2}
 Let $P$ be a special spine of a hyperbolic $3$-manifold $M$ with totally geodesic boundary such that $\operatorname{c}(M) \neq 0,1$. Then   $\operatorname{c}(M) = \operatorname{cv} [P]$.
\end{corollary}

 The next result follows from Theorem~\ref{thm:main} and Remark~\ref{cor2}.

\begin{theorem} \label{theorem9}
  Let $P$ be a special spine of a hyperbolic $3$-manifold $M$ such that $P$ has $n \geq 2$ true vertices and two $2$-components. If the boundary curves of both $2$-components of $P$ are not short, then   $\operatorname{c} (M) = n$.
\end{theorem}

 The complexity of manifolds which have special spines with $n \geq 2$ true vertices and only one $2$-component was calculated in \cite{FrigMarPet03}  in terms of ideal triangulations. We reformulate this result in terms of spines.

\begin{theorem}  \cite[Theorem 1.2]{FrigMarPet03} \label{theorem7}
 Let $P$ be a special spine of a $3$-manifold $M$ such that $P$ has $n \geq 2$ true vertices and one $2$-component. Then   $\operatorname{c} (M) = n$.
\end{theorem}

\section{Two examples}
 \label{section6}

 We give two examples of infinite series of $3$-manifolds and their special spines  satisfying the conditions of Theorem \ref{theorem9}. The complexity of these manifolds was   previously computed by the authors in \cite{VesFom11} and \cite{Stekl2015}.

\subsection{Paoluzzi--Zimmermann manifolds}

We describe a two-parameter family of $3$-manifolds $M_{n,k}$ with nonempty boundary, constructed by Paoluzzi and Zimmermann  \cite{PaolZim96}.  For an integer $n\geqslant 4$ consider an $n$-gonal bipyramid $\mathcal B_{n}$, which is the union of pyramids
$$N V_{0} V_{1} \ldots V_{n-1} \qquad \text{ and } \qquad S V_{0} V_{1} \ldots V_{n-1}$$
meeting each other along the common $n$-gonal base
$V_{0} V_{1} \ldots V_{n-1}$. Let $k$ be an integer such that $0 \leqslant k < n$ and $gcd(n, $2$-k)=1$. For each $i = 0, \ldots, n-1$ consider a transformation $y_{i}$ which identifies the face $V_{i} V_{i+1} N$ with
the face $S V_{i+k} V_{i+k+1}$ (indices are taken mod $n$ and the
vertices are glued to each other in the order in which they
are written).

\begin{figure}[htb]
\centering
\unitlength=0.4mm
\begin{picture}(180,130)(-90,-20)
\thicklines
\put(-120,40){\line(2,1){120}}
\put(-60,40){\line(1,1){60}}
\put(-20,40){\line(1,3){20}}
\put(0,40){\line(0,1){60}}
\put(20,40){\line(-1,3){20}}
\put(60,40){\line(-1,1){60}}
\put(120,40){\line(-2,1){120}}
\put(-120,40){\vector(2,1){60}}
\put(-60,40){\vector(1,1){30}}
\put(-20,40){\vector(1,3){10}}
\put(0,40){\vector(0,1){30}}
\put(20,40){\vector(-1,3){10}}
\put(60,40){\vector(-1,1){30}}
\put(120,40){\vector(-2,1){60}}
\put(0,-20){\line(-2,1){120}}
\put(0,-20){\line(-1,1){60}}
\put(0,-20){\line(-1,3){20}}
\put(0,-20){\line(0,1){60}}
\put(0,-20){\line(1,3){20}}
\put(0,-20){\line(1,1){60}}
\put(0,-20){\line(2,1){120}}
\put(0,-20){\vector(2,1){60}}
\put(0,-20){\vector(1,1){30}}
\put(0,-20){\vector(1,3){10}}
\put(0,-20){\vector(0,1){30}}
\put(0,-20){\vector(-1,3){10}}
\put(0,-20){\vector(-1,1){30}}
\put(0,-20){\vector(-2,1){60}}
\put(-120,40){\line(1,0){60}}
\put(-60,40){\line(1,0){40}}
\put(-20,40){\line(1,0){20}}
\put(0,40){\line(1,0){20}}
\put(20,40){\line(1,0){40}}
\put(60,40){\line(1,0){60}}
\put(-120,40){\vector(1,0){35}}
\put(-60,40){\vector(1,0){25}}
\put(-20,40){\vector(1,0){15}}
\put(0,40){\vector(1,0){10}}
\put(20,40){\vector(1,0){20}}
\put(60,40){\vector(1,0){30}}
\put(5,105){\makebox(0,0)[cc]{$N$}}
\put(5,-25){\makebox(0,0)[cc]{$S$}}
\put(-125,45){\makebox(0,0)[cc]{$V_{5}$}}
\put(-65,45){\makebox(0,0)[cc]{$V_{0}$}}
\put(-25,45){\makebox(0,0)[cc]{$V_{1}$}}
\put(-5,45){\makebox(0,0)[cc]{$V_{2}$}}
\put(25,45){\makebox(0,0)[cc]{$V_{3}$}}
\put(65,45){\makebox(0,0)[cc]{$V_{4}$}}
\put(125,45){\makebox(0,0)[cc]{$V_{5}$}}
\end{picture}
\caption{The bipyramid ${\mathcal B}_{6}$.} \label{fig:B6}
\end{figure}
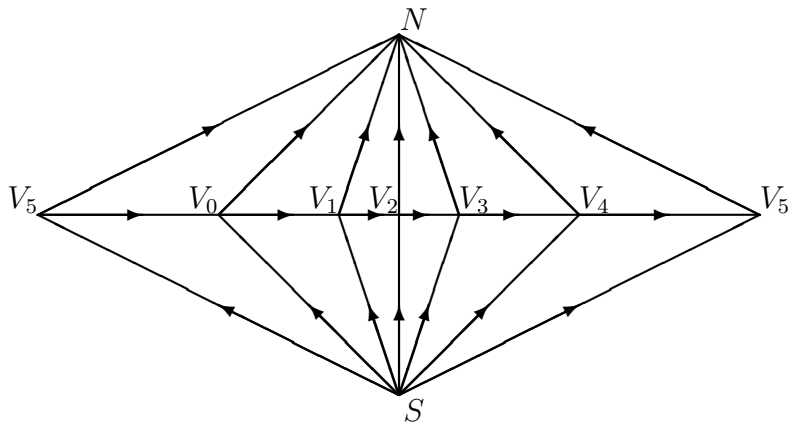

 The identifications $\{ y_{0}, y_{1}, \ldots , y_{n-1} \}$ define  
equivalence relations on the sets of faces, edges, and vertices of
the bipyramid. It is easy to see that   the faces are partitioned into
pairs of equivalent faces, the edges constitute one equivalence class and so do the vertices  (this is guaranteed by the above
conditions on~$k$). Denote the associated quotient space  by
$M_{n,k}^{*}$. It is an orientable pseudomanifold with one singular
point, since $\chi(M_{n,k}^{*}) = 1 - 1 + n - 1 = n-1 \neq 0$.
Cutting off an open cone neighborhood of the singular point of
$M_{n,k}^{*}$ we obtain a compact manifold $M_{n,k}$ with one boundary
component.

 Let us construct a special spine $P_{n,k}$ of $M_{n,k}$. Cut $\mathcal B_{n}$ into $n$ tetrahedra $\mathcal T_{i} = N S V_{i} V_{i+1}$, where $i = 0, 1, \ldots, n-1$. For each $\mathcal T_{i}$ consider the union $R_i$ of the links of all four vertices of $\mathcal T_{i}$ in the first barycentric subdivision. The pseudomanifold $M_{n,k}^{*}$ can be obtained by gluing the tetrahedra $\mathcal T_{0}, \ldots , \mathcal T_{n-1}$ via the identifications $\{ y_{0}, y_{1}, \ldots , y_{n-1} \}$. This gluing determines a pseudotriangulation $\mathcal T$ of $M_{n,k}^{*}$ and induces a gluing of the corresponding polyhedra $R_i$, $i = 0, \ldots, n-1$, together. This gluing yields a special spine $P_{n,k} = \cup_i R_i$ of $M_{n,k}$. Since every $R_i$ is homeomorphic to a cone over $K_4$, the spine $P_{n,k}$ has exactly $n$ true vertices.

\begin{corollary} \cite{VesFom11}
For every integer $n \geqslant 4$ we have $c(M_{n,k})=n$.
\end{corollary}

\begin{proof}
  By the construction of the special spine $P_{n,k}$, its $2$-components are in a one-to-one correspondence with the edges of the pseudotriangulation $\mathcal T$. Since $\mathcal T$ has two edges, $P_{n,k}$ has two $2$-components. In addition, the boundary curves of   both these $2$-components   are not short, because each of them traverses all $n\geqslant 4 $ true vertices of the spine.

  Paoluzzi and Zimmermann  \cite{PaolZim96} proved that the manifolds  $M_{n,k}$ are hyperbolic with totally geodesic boundary. Thus the manifolds $M_{n,k}$ and their special spines $P_{n,k}$ satisfy the conditions of Theorem \ref{theorem9}. This implies the corollary.
\end{proof}

\subsection{Manifolds from \cite{Stekl2015}}

We describe a family of $3$-manifolds $N_{n}$ with nonempty boundary constructed in~\cite{Stekl2015}. Let $s$ be a nonnegative integer and $n = 5+4s$. We construct a plane $4$-regular graph $G_{n}$ with decoration of vertices and edges as follows. The graph $G_{n}$ has $n$ vertices, two loops, and $n-1$ double edges. At each vertex of the graph, over- and under- passing are specified (just as in a crossing in a knot diagram), and each edge is assigned an element of the cyclic group $\mathbb Z_{3} = \{ 0,1,2 \}$. The decorated graph $G_{5}$ is shown in Fig.~\ref{figG5}.

\begin{figure}[h]
\begin{center}
\special{em:linewidth 1.4pt} \unitlength=.34mm
\begin{picture}(0,50)(0,0)
\thicklines
\put(-80,20){\circle*{3}}
\put(-40,20){\circle*{3}}
\put(0,20){\circle*{3}}
\put(40,20){\circle*{3}}
\put(80,20){\circle*{3}}
\qbezier(-80,20)(-60,40)(-43,23)
\qbezier(-77,17)(-60,0)(-40,20)
\qbezier(-40,20)(-20,40)(-3,23)
\qbezier(-37,17)(-20,0)(0,20)
\qbezier(0,20)(20,40)(37,23)
\qbezier(3,17)(20,0)(40,20)
\qbezier(40,20)(60,40)(77,23)
\qbezier(43,17)(60,0)(80,20)
\qbezier(-80,20)(-80,10)(-90,10)
\qbezier(-90,10)(-100,10)(-100,20)
\qbezier(-100,20)(-100,30)(-90,30)
\qbezier(-90,30)(-83,30)(-81,24)
\qbezier(80,20)(80,30)(90,30)
\qbezier(90,30)(100,30)(100,20)
\qbezier(100,20)(100,10)(90,10)
\qbezier(90,10)(83,10)(81,16)
\qbezier(-40,40)(-40,37)(-40,35)
\qbezier(-40,30)(-40,27)(-40,25)
\qbezier(-40,15)(-40,15)(-40,10)
\qbezier(-40,5)(-40,5)(-40,0)
\qbezier(40,40)(40,37)(40,35)
\qbezier(40,30)(40,27)(40,25)
\qbezier(40,15)(40,15)(40,10)
\qbezier(40,5)(40,5)(40,0)
\put(-105,20){\makebox(0,0)[c]{$1$}}
\put(105,20){\makebox(0,0)[c]{$1$}}
\put(-60,35){\makebox(0,0)[c]{$0$}}
\put(-60,5){\makebox(0,0)[c]{$1$}}
\put(-20,35){\makebox(0,0)[c]{$1$}}
\put(-20,5){\makebox(0,0)[c]{$1$}}
\put(20,35){\makebox(0,0)[c]{$0$}}
\put(20,5){\makebox(0,0)[c]{$0$}}
\put(60,35){\makebox(0,0)[c]{$0$}}
\put(60,5){\makebox(0,0)[c]{$1$}}
\end{picture}
\end{center} \caption{The decorated graph $G_{5}$.} \label{figG5}
\end{figure}
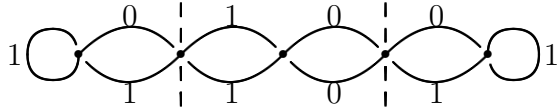

The graph $G_{5}$ has a block structure: it is composed of three subgraphs $A$, $C$, and $E$ shown in Fig.~\ref{figG5block}. We express this fact as $G_{5} = A \cdot C \cdot E$. Each of the graphs $A$ and $E$ has one $4$-valent vertex and one $2$-valent vertex. The graph $C$ has one $4$-valent vertex and two $2$-valent vertices. The decorations of the vertices and edges of the graphs $A$, $C$, and $E$ are induced by the decoration of
the graph $G_{5}$.

\begin{figure}[h]
\begin{center}
\special{em:linewidth 1.4pt} \unitlength=.34mm
\begin{picture}(0,60)(0,-10)
\thicklines
\put(-40,20){\circle*{3}}
\put(0,20){\circle*{3}}
\put(40,20){\circle*{3}}
\qbezier(-40,20)(-20,40)(-3,23)
\qbezier(-37,17)(-20,0)(0,20)
\qbezier(0,20)(20,40)(37,23)
\qbezier(3,17)(20,0)(40,20)
\put(-80,20){\circle*{3}}
\put(-120,20){\circle*{3}}
\qbezier(-120,20)(-120,10)(-130,10)
\qbezier(-130,10)(-140,10)(-140,20)
\qbezier(-140,20)(-140,30)(-130,30)
\qbezier(-130,30)(-123,30)(-121,24)
\qbezier(-120,20)(-100,40)(-83,23)
\qbezier(-117,17)(-100,0)(-80,20)
\put(-120,-10){\makebox(0,0)[c]{$A$}}
\put(-145,20){\makebox(0,0)[c]{$1$}}
\put(-100,35){\makebox(0,0)[c]{$0$}}
\put(-100,5){\makebox(0,0)[c]{$1$}}
\put(80,20){\circle*{3}}
\put(120,20){\circle*{3}}
\qbezier(120,20)(120,30)(130,30)
\qbezier(130,30)(140,30)(140,20)
\qbezier(140,20)(140,10)(130,10)
\qbezier(130,10)(123,10)(121,16)
\put(145,20){\makebox(0,0)[c]{$1$}}
\put(120,-10){\makebox(0,0)[c]{$E$}}
\put(100,35){\makebox(0,0)[c]{$0$}}
\put(100,5){\makebox(0,0)[c]{$1$}}
\qbezier(80,20)(100,40)(117,23)
\qbezier(83,17)(100,0)(120,20)
\put(0,-10){\makebox(0,0)[c]{$C$}}
\put(-20,35){\makebox(0,0)[c]{$1$}}
\put(-20,5){\makebox(0,0)[c]{$1$}}
\put(20,35){\makebox(0,0)[c]{$0$}}
\put(20,5){\makebox(0,0)[c]{$0$}}
\end{picture}
\end{center} \caption{Subgraphs $A$, $C$, and $E$ of $G_{5}$.} \label{figG5block}
\end{figure}
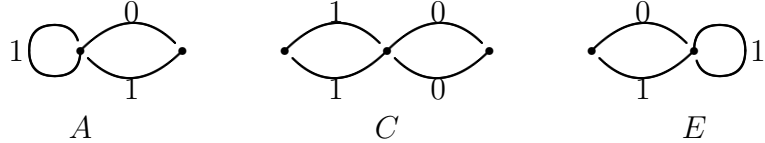

Next, we define graphs $B$ and $D$ as shown in Fig. ~\ref{figBDblocks}. The graphs $B$ and $D$ have the same combinatorial structure as the graph $C$ and the same decorations of vertices; however, they differ by the decorations
of edges.

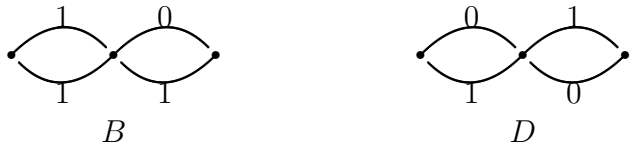
\begin{figure}[h]
\begin{center}
\special{em:linewidth 1.4pt} \unitlength=.34mm
\begin{picture}(0,60)(0,-10)
\thicklines
\put(-120,20){\circle*{3}}
\put(-80,20){\circle*{3}}
\put(-40,20){\circle*{3}}
\qbezier(-120,20)(-100,40)(-83,23)
\qbezier(-117,17)(-100,0)(-80,20)
\qbezier(-80,20)(-60,40)(-43,23)
\qbezier(-77,17)(-60,0)(-40,20)
\put(-80,-10){\makebox(0,0)[c]{$B$}}
\put(-100,35){\makebox(0,0)[c]{$1$}}
\put(-100,5){\makebox(0,0)[c]{$1$}}
\put(-60,35){\makebox(0,0)[c]{$0$}}
\put(-60,5){\makebox(0,0)[c]{$1$}}
\put(40,20){\circle*{3}}
\put(80,20){\circle*{3}}
\put(120,20){\circle*{3}}
\qbezier(40,20)(60,40)(77,23)
\qbezier(43,17)(60,0)(80,20)
\qbezier(80,20)(100,40)(117,23)
\qbezier(83,17)(100,0)(120,20)
\put(80,-10){\makebox(0,0)[c]{$D$}}
\put(60,35){\makebox(0,0)[c]{$0$}}
\put(60,5){\makebox(0,0)[c]{$1$}}
\put(100,35){\makebox(0,0)[c]{$1$}}
\put(100,5){\makebox(0,0)[c]{$0$}}
%
\end{picture}
\end{center} \caption{Graphs $B$ and $D$.} \label{figBDblocks}
\end{figure}

Let $G_{n}$ be the decorated graph    composed successively of the subgraph $A$, $s$ copies of the subgraph $B$, the subgraph $C$, $s$ copies of the subgraph $D$, and the subgraph $E$. In other words,
$G_{n} = A \cdot B^{s} \cdot C \cdot D^{s} \cdot E$. The graph $G_{9}$ is shown in Fig.~\ref{figG9}.

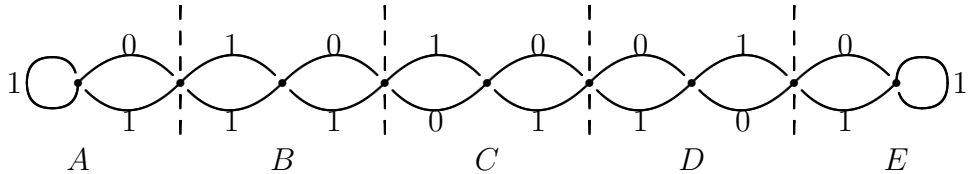
\begin{figure}[h]
\begin{center}
\special{em:linewidth 1.4pt} \unitlength=.34mm
\begin{picture}(0,60)(0,-10)
\thicklines
\put(-160,20){\circle*{3}}
\put(-120,20){\circle*{3}}
\put(-80,20){\circle*{3}}
\put(-40,20){\circle*{3}}
\put(0,20){\circle*{3}}
\put(40,20){\circle*{3}}
\put(80,20){\circle*{3}}
\put(120,20){\circle*{3}}
\put(160,20){\circle*{3}}
\qbezier(-160,20)(-140,40)(-123,23)
\qbezier(-157,17)(-140,0)(-120,20)
\qbezier(-120,20)(-100,40)(-83,23)
\qbezier(-117,17)(-100,0)(-80,20)
\qbezier(-80,20)(-60,40)(-43,23)
\qbezier(-77,17)(-60,0)(-40,20)
\qbezier(-40,20)(-20,40)(-3,23)
\qbezier(-37,17)(-20,0)(0,20)
\qbezier(0,20)(20,40)(37,23)
\qbezier(3,17)(20,0)(40,20)
\qbezier(40,20)(60,40)(77,23)
\qbezier(43,17)(60,0)(80,20)
\qbezier(80,20)(100,40)(117,23)
\qbezier(83,17)(100,0)(120,20)
\qbezier(120,20)(140,40)(157,23)
\qbezier(123,17)(140,0)(160,20)
\qbezier(-160,20)(-160,10)(-170,10)
\qbezier(-170,10)(-180,10)(-180,20)
\qbezier(-180,20)(-180,30)(-170,30)
\qbezier(-170,30)(-163,30)(-161,24)
\qbezier(160,20)(160,30)(170,30)
\qbezier(170,30)(180,30)(180,20)
\qbezier(180,20)(180,10)(170,10)
\qbezier(170,10)(163,10)(161,16)
\qbezier(-120,50)(-120,50)(-120,45)
\qbezier(-120,40)(-120,37)(-120,35)
\qbezier(-120,30)(-120,27)(-120,25)
\qbezier(-120,15)(-120,15)(-120,10)
\qbezier(-120,5)(-120,5)(-120,0)
\qbezier(-40,50)(-40,50)(-40,45)
\qbezier(-40,40)(-40,37)(-40,35)
\qbezier(-40,30)(-40,27)(-40,25)
\qbezier(-40,15)(-40,15)(-40,10)
\qbezier(-40,5)(-40,5)(-40,0)
\qbezier(40,50)(40,50)(40,45)
\qbezier(40,40)(40,37)(40,35)
\qbezier(40,30)(40,27)(40,25)
\qbezier(40,15)(40,15)(40,10)
\qbezier(40,5)(40,5)(40,0)
\qbezier(120,50)(120,50)(120,45)
\qbezier(120,40)(120,37)(120,35)
\qbezier(120,30)(120,27)(120,25)
\qbezier(120,15)(120,15)(120,10)
\qbezier(120,5)(120,5)(120,0)
\put(-160,-10){\makebox(0,0)[c]{$A$}}
\put(-80,-10){\makebox(0,0)[c]{$B$}}
\put(0,-10){\makebox(0,0)[c]{$C$}}
\put(80,-10){\makebox(0,0)[c]{$D$}}
\put(160,-10){\makebox(0,0)[c]{$E$}}
\put(-185,20){\makebox(0,0)[c]{$1$}}
\put(185,20){\makebox(0,0)[c]{$1$}}
\put(-140,35){\makebox(0,0)[c]{$0$}}
\put(-140,5){\makebox(0,0)[c]{$1$}}
\put(-100,35){\makebox(0,0)[c]{$1$}}
\put(-100,5){\makebox(0,0)[c]{$1$}}
\put(-60,35){\makebox(0,0)[c]{$0$}}
\put(-60,5){\makebox(0,0)[c]{$1$}}
\put(-20,35){\makebox(0,0)[c]{$1$}}
\put(-20,5){\makebox(0,0)[c]{$0$}}
\put(20,35){\makebox(0,0)[c]{$0$}}
\put(20,5){\makebox(0,0)[c]{$1$}}
\put(60,35){\makebox(0,0)[c]{$0$}}
\put(60,5){\makebox(0,0)[c]{$1$}}
\put(100,35){\makebox(0,0)[c]{$1$}}
\put(100,5){\makebox(0,0)[c]{$0$}}
\put(140,35){\makebox(0,0)[c]{$0$}}
\put(140,5){\makebox(0,0)[c]{$1$}}
\end{picture}
\end{center} \caption{The decorated graph $G_{9}$.} \label{figG9}
\end{figure}

Note that the graphs $G_{n}$ belong to the class of o-graphs defined in \cite{B-P}. That paper presents an algorithm   constructing a special polyhedron from an arbitrary o-graph. According to the algorithm, to obtain the polyhedron determined by   $G_{n}$ one should replace the subgraphs $A$, $B$, $C$, $D$, and $E$ by the similarly named blocks shown in Figs.~\ref{fig101} and~\ref{fig102}. As a result of such gluing of
blocks, we obtain a special polyhedron. It is proved in \cite{B-P} that this polyhedron is a special spine of a compact orientable $3$-manifold with nonempty boundary. Let $N_{n}$ and $P_{n}$ be a manifold and its special spine, respectively,   constructed from the o-graph $G_{n}$ via this algorithm from \cite{B-P}.

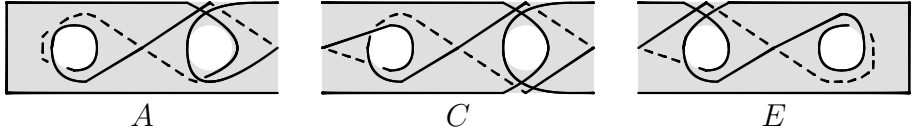
\begin{figure}[h]
\begin{center}
\special{em:linewidth 1.4pt} \unitlength=.3 mm
\begin{picture}(0,50)(0,0)
\thicklines
\put(-140,0){\begin{picture}(0,50)(0,-10)
\put(-60,0){\begin{tikzpicture}
  \begin{scope}
    \fill[gray!25]
    (0mm,0mm) rectangle (18mm,12mm) ;
      \clip (9mm,6mm) circle (3mm);
    \fill[white] (0,0) rectangle (18mm,12mm);
	  \end{scope}
\end{tikzpicture}
}
\put(0,0){\begin{tikzpicture}
  \begin{scope}
    \fill[gray!25]
    (0mm,0mm) rectangle (18mm,12mm) ;
      \clip (9mm,6mm) circle (3mm);
    \fill[white] (0,0) rectangle (18mm,12mm);
	  \end{scope}
\end{tikzpicture}
}
\qbezier(-60,0)(-60,0)(60,0)
\qbezier(-60,40)(-60,40)(20,40)
\qbezier(28,7)(60,20)(20,40)
\qbezier(-60,0)(-60,0)(-60,40)
\qbezier(-40,20)(-40,30)(-30,30)
\qbezier(-30,30)(-20,30)(-20,20)
\qbezier(-20,20)(-20,10)(-30,10)
\qbezier(-30,10)(-30,10)(-33,11)
\qbezier(-42,14)(-42,14)(-44,15)
\qbezier(-44,18)(-44,18)(-44,22)
\qbezier(-44,25)(-44,25)(-42,27)
\qbezier(30,5)(40,5)(60,20)
\qbezier(30,5)(20,5)(20,20)
\qbezier(50,40)(20,40)(20,20)
\qbezier(60,40)(60,40)(50,40)
\qbezier(24,5)(24,5)(21,6)
\qbezier(18,8)(18,8)(15,10)
\qbezier(12,12)(12,12)(9,14)
\qbezier(6,16)(6,16)(3,18)
\qbezier(30,40)(30,40)(31.5,39)
\qbezier(36,36)(36,36)(39,34)
\qbezier(42,32)(42,32)(45,30)
\qbezier(48,28)(48,28)(51,26)
\qbezier(54,24)(54,24)(57,22)
\qbezier(30,40)(30,40)(0,20)
\qbezier(0,20)(0,20)(-25,5)
\qbezier(-40,20)(-40,5)(-25,5)
\qbezier(-40,20)(-40,20)(-39,25)
\qbezier(-3,22)(-3,22)(-6,24)
\qbezier(-9,26)(-9,26)(-12,28)
\qbezier(-15,30)(-15,30)(-18,32)
\qbezier(-21,34)(-21,34)(-24,35)
\qbezier(-27,35)(-27,35)(-30,34)
\qbezier(-33,33)(-33,33)(-36,31)
\put(0,-10){\makebox(0,0)[c]{$A$}}
\end{picture}}
\put(0,0){\begin{picture}(0,50)(0,-10)
\put(-60,0){\begin{tikzpicture}
  \begin{scope}
    \fill[gray!25]
    (0mm,0mm) rectangle (18mm,12mm) ;
      \clip (9mm,6mm) circle (3mm);
    \fill[white] (0,0) rectangle (18mm,12mm);
	  \end{scope}
\end{tikzpicture}
}
\put(0,0){\begin{tikzpicture}
  \begin{scope}
    \fill[gray!25]
    (0mm,0mm) rectangle (18mm,12mm) ;
      \clip (9mm,6mm) circle (3mm);
    \fill[white] (0,0) rectangle (18mm,12mm);
	  \end{scope}
\end{tikzpicture}
}
\qbezier(-60,0)(-60,0)(20,0)
\qbezier(-60,40)(-60,40)(20,40)
\qbezier(20,0)(60,20)(20,40)
\qbezier(-60,20)(-60,20)(-30,30)
\qbezier(-30,30)(-20,30)(-20,20)
\qbezier(-20,20)(-20,10)(-30,10)
\qbezier(-30,10)(-30,10)(-33,11)
\qbezier(-42,14)(-42,14)(-45,15)
\qbezier(-48,16)(-48,16)(-51,17)
\qbezier(-54,18)(-54,18)(-57,19)
\qbezier(60,0)(60,0)(50,0)
\qbezier(50,0)(20,0)(20,20)
\qbezier(50,40)(20,40)(20,20)
\qbezier(60,40)(60,40)(50,40)
\qbezier(60,20)(60,20)(30,0)
\qbezier(30,0)(30,0)(28.5,1)
\qbezier(24,4)(24,4)(21,6)
\qbezier(18,8)(18,8)(15,10)
\qbezier(12,12)(12,12)(9,14)
\qbezier(6,16)(6,16)(3,18)
\qbezier(30,40)(30,40)(31.5,39)
\qbezier(36,36)(36,36)(39,34)
\qbezier(42,32)(42,32)(45,30)
\qbezier(48,28)(48,28)(51,26)
\qbezier(54,24)(54,24)(57,22)
\qbezier(30,40)(30,40)(0,20)
\qbezier(0,20)(0,20)(-25,5)
\qbezier(-40,20)(-40,5)(-25,5)
\qbezier(-40,20)(-40,20)(-39,25)
\qbezier(-3,22)(-3,22)(-6,24)
\qbezier(-9,26)(-9,26)(-12,28)
\qbezier(-15,30)(-15,30)(-18,32)
\qbezier(-21,34)(-21,34)(-24,35)
\qbezier(-27,35)(-27,35)(-30,34)
\qbezier(-33,32)(-33,32)(-36,30)
\put(0,-10){\makebox(0,0)[c]{$C$}}
\end{picture}}
\put(140,0){\begin{picture}(0,50)(0,-10)
\thicklines
\put(-60,0){\begin{tikzpicture}
  \begin{scope}
    \fill[gray!25]
    (0mm,0mm) rectangle (18mm,12mm) ;
      \clip (9mm,6mm) circle (3mm);
    \fill[white] (0,0) rectangle (18mm,12mm);
	  \end{scope}
\end{tikzpicture}
}
\put(0,0){\begin{tikzpicture}
  \begin{scope}
    \fill[gray!25]
    (0mm,0mm) rectangle (18mm,12mm) ;
      \clip (9mm,6mm) circle (3mm);
    \fill[white] (0,0) rectangle (18mm,12mm);
	  \end{scope}
\end{tikzpicture}
}
\qbezier(-60,40)(-60,40)(-40,40)
\qbezier(-40,40)(-20,30)(-20,20)
\qbezier(-20,20)(-20,10)(-30,10)
\qbezier(-30,10)(-30,10)(-33,11)
\qbezier(-42,14)(-42,14)(-45,15)
\qbezier(-48,16)(-48,16)(-51,17)
\qbezier(-54,18)(-54,18)(-57,19)
\qbezier(60,40)(60,40)(60,0)
\qbezier(60,0)(60,0)(-60,0)
\qbezier(33,30)(20,30)(20,20)
\qbezier(20,20)(20,10)(30,10)
\qbezier(30,10)(40,10)(40,20)
\qbezier(40,20)(40,35)(30,35)
\qbezier(44,23)(44,23)(43,26)
\qbezier(44,17)(44,17)(44,20)
\qbezier(44,14)(44,14)(44,11)
\qbezier(42,9)(42,9)(39,7)
\qbezier(36,6)(36,6)(33,5)
\qbezier(30,5)(30,5)(27,5)
\qbezier(24,5)(24,5)(21,6)
\qbezier(18,8)(18,8)(15,10)
\qbezier(12,12)(12,12)(9,14)
\qbezier(6,16)(6,16)(3,18)
\qbezier(30,35)(30,35)(0,20)
\qbezier(0,20)(0,20)(-25,5)
\qbezier(-40,20)(-40,5)(-25,5)
\qbezier(-40,20)(-40,20)(-39,25)
\qbezier(-3,22)(-3,22)(-6,24)
\qbezier(-9,26)(-9,26)(-12,28)
\qbezier(-15,30)(-15,30)(-18,32)
\qbezier(-21,34)(-21,34)(-24,36)
\qbezier(-27,38)(-27,38)(-30,40)
\qbezier(-30,40)(-30,40)(-60,20)
\qbezier(-40,20)(-40,30)(-20,40)
\qbezier(-20,40)(-20,40)(60,40)
\put(0,-10){\makebox(0,0)[c]{$E$}}
\end{picture}}
\end{picture}
\end{center} \caption{Blocks $A$, $C$, and $E$.} \label{fig101}
\end{figure}

\begin{figure}[h]
\begin{center}
\special{em:linewidth 1.4pt} \unitlength=.3mm
\begin{picture}(0,50)(0,0)
\put(-80,0){\begin{picture}(0,50)(0,-10)
\thicklines
\put(-60,0){\begin{tikzpicture}
  \begin{scope}
    \fill[gray!25]
    (0mm,0mm) rectangle (18mm,12mm) ;
      \clip (9mm,6mm) circle (3mm);
    \fill[white] (0,0) rectangle (18mm,12mm);
	  \end{scope}
\end{tikzpicture}
}
\put(0,0){\begin{tikzpicture}
  \begin{scope}
    \fill[gray!25]
    (0mm,0mm) rectangle (18mm,12mm) ;
      \clip (9mm,6mm) circle (3mm);
    \fill[white] (0,0) rectangle (18mm,12mm);
	  \end{scope}
\end{tikzpicture}
}
\qbezier(-60,0)(-60,0)(60,0)
\qbezier(-60,40)(-60,40)(20,40)
\qbezier(28,7)(60,20)(20,40)
\qbezier(-60,20)(-60,20)(-30,30)
\qbezier(-30,30)(-20,30)(-20,20)
\qbezier(-20,20)(-20,10)(-30,10)
\qbezier(-30,10)(-30,10)(-33,11)
\qbezier(-42,14)(-42,14)(-45,15)
\qbezier(-48,16)(-48,16)(-51,17)
\qbezier(-54,18)(-54,18)(-57,19)
\qbezier(30,5)(40,5)(60,20)
\qbezier(30,5)(20,5)(20,20)
\qbezier(50,40)(20,40)(20,20)
\qbezier(60,40)(60,40)(50,40)
\qbezier(24,5)(24,5)(21,6)
\qbezier(18,8)(18,8)(15,10)
\qbezier(12,12)(12,12)(9,14)
\qbezier(6,16)(6,16)(3,18)
\qbezier(30,40)(30,40)(31.5,39)
\qbezier(36,36)(36,36)(39,34)
\qbezier(42,32)(42,32)(45,30)
\qbezier(48,28)(48,28)(51,26)
\qbezier(54,24)(54,24)(57,22)
\qbezier(30,40)(30,40)(0,20)
\qbezier(0,20)(0,20)(-25,5)
\qbezier(-40,20)(-40,5)(-25,5)
\qbezier(-40,20)(-40,20)(-39,25)
\qbezier(-3,22)(-3,22)(-6,24)
\qbezier(-9,26)(-9,26)(-12,28)
\qbezier(-15,30)(-15,30)(-18,32)
\qbezier(-21,34)(-21,34)(-24,35)
\qbezier(-27,35)(-27,35)(-30,34)
\qbezier(-33,32)(-33,32)(-36,30)
\end{picture}
\put(0,-10){\makebox(0,0)[c]{$B$}}
}
\put(80,0){\begin{picture}(0,50)(0,-10)
\thicklines
\put(-60,0){\begin{tikzpicture}
  \begin{scope}
    \fill[gray!25]
    (0mm,0mm) rectangle (18mm,12mm) ;
      \clip (9mm,6mm) circle (3mm);
    \fill[white] (0,0) rectangle (18mm,12mm);
	  \end{scope}
\end{tikzpicture}
}
\put(0,0){\begin{tikzpicture}
  \begin{scope}
    \fill[gray!25]
    (0mm,0mm) rectangle (18mm,12mm) ;
      \clip (9mm,6mm) circle (3mm);
    \fill[white] (0,0) rectangle (18mm,12mm);
	  \end{scope}
\end{tikzpicture}
}
\qbezier(-60,40)(-60,40)(-40,40)
\qbezier(-40,40)(-20,30)(-20,20)
\qbezier(-20,20)(-20,10)(-30,10)
\qbezier(-30,10)(-30,10)(-33,11)
\qbezier(-42,14)(-42,14)(-45,15)
\qbezier(-48,16)(-48,16)(-51,17)
\qbezier(-54,18)(-54,18)(-57,19)
\qbezier(60,0)(60,0)(50,0)
\qbezier(50,0)(20,0)(20,20)
\qbezier(33,30)(20,30)(20,20)
\qbezier(42,29)(42,29)(44,28)
\qbezier(46,27)(46,27)(48,26)
\qbezier(50,25)(50,25)(52,24)
\qbezier(54,23)(54,23)(56,22)
\qbezier(60,20)(60,20)(30,0)
\qbezier(30,0)(30,0)(28.5,1)
\qbezier(24,4)(24,4)(21,6)
\qbezier(18,8)(18,8)(15,10)
\qbezier(12,12)(12,12)(9,14)
\qbezier(6,16)(6,16)(3,18)
\qbezier(30,35)(30,35)(0,20)
\qbezier(30,35)(55,20)(25,0)
\qbezier(-60,0)(-60,0)(25,0)
\qbezier(0,20)(0,20)(-25,5)
\qbezier(-40,20)(-40,5)(-25,5)
\qbezier(-40,20)(-40,20)(-39,25)
\qbezier(-3,22)(-3,22)(-6,24)
\qbezier(-9,26)(-9,26)(-12,28)
\qbezier(-15,30)(-15,30)(-18,32)
\qbezier(-21,34)(-21,34)(-24,36)
\qbezier(-27,38)(-27,38)(-30,40)
\qbezier(-30,40)(-30,40)(-60,20)
\qbezier(-40,20)(-40,30)(-20,40)
\qbezier(-20,40)(-20,40)(60,40)
\put(0,-10){\makebox(0,0)[c]{$D$}}
\end{picture}}
\end{picture}
\end{center} \caption{Blocks $B$ and $D$.} \label{fig102}
\end{figure}
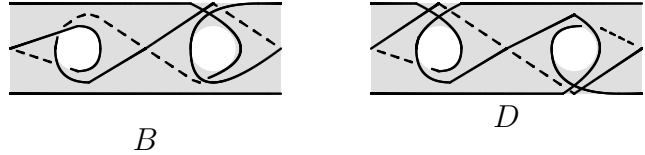

\begin{corollary} \cite{Stekl2015}
For every n = 5+4s, where s is a nonnegative integer, the complexity of the manifold $N_{n}$ is equal to $n$.
\end{corollary}

\begin{proof}
 According to the algorithm above, the number of true vertices of the spine $P_{n}$ is equal to the number of vertices of the o-graph $G_{n}$, i.e., to $n$. The curves shown in the pictures of the blocks are joined into closed curves. It can be easily verified that for any $n$, the number of closed curves is equal to two. Since these closed curves correspond to the $2$-components of the spine $P_{n}$, we obtain that $P_{n}$ has two  $2$-components. It follows from the construction of the spine that the boundary curves of both $2$-components of $P_{n}$ are not short. Moreover, it is proved in \cite{Stekl2015} that the manifolds $N_{n}$ are hyperbolic manifolds with totally geodesic boundary. Therefore, the manifolds $N_{n}$ and their special spines $P_{n}$ satisfy the conditions of Theorem~\ref{theorem9}. This proves the corollary.
\end{proof}

\end{document}